\documentclass[a4paper,11pt]{amsart}

\usepackage{amssymb}
\usepackage{mhequ}
\usepackage[usenames]{color}
\newcommand{\red}{}

\usepackage{cite}

\DeclareFontFamily{OT1}{rsfs}{}
\DeclareFontShape{OT1}{rsfs}{m}{n}{ <-7> rsfs5 <7-10> rsfs7 <10->
rsfs10}{} \DeclareMathAlphabet{\mathscr}{OT1}{rsfs}{m}{n}
\newcommand{\eq}[1]{\eqref{#1}}

\newcommand{\bel}[1]{\begin{equation}\label{#1}}
\newcommand{\beal}[1]{\begin{eqnarray}\label{#1}}
\newcommand{\beadl}[1]{\begin{deqarr}\label{#1}}
\newcommand{\eeadl}[1]{\arrlabel{#1}\end{deqarr}}
\newcommand{\eeal}[1]{\label{#1}\end{eqnarray}}
\newcommand{\eead}[1]{\end{deqarr}}
\newcommand{\eea}{\end{eqnarray}}
\newcommand{\eeaa}{\end{eqnarray*}}

\newcommand{\be}{\begin{equation}}
\newcommand{\ee}{\end{equation}}

\DeclareFontFamily{OT1}{rsfs}{}
\DeclareFontShape{OT1}{rsfs}{m}{n}{ <-7> rsfs5 <7-10> rsfs7 <10->
rsfs10}{} \DeclareMathAlphabet{\mycal}{OT1}{rsfs}{m}{n}

\newcounter{mnotecount}[section]

\newcommand{\N}{{\Bbb N}}

\newcommand{\rmnote}[1]{}%{\mnote{#1}}

\newcommand{\Ric}{\operatorname{Ric}}

\newcommand{\zH}{\mathring H}

\newcommand{\maclKzo}{{\mathcal K}^{\bot}}
\newcommand{\Hpsi}[1]{\zH^{#1}_{\phi,\psi}}

\def\mysavedown#1{\edef\mysubs{\mysubs#1}}
\def\mysaveup#1{\edef\mysups{\mysups#1}}
\def\mydown#1{{\mytensor}_{\vphantom{\mysubs}#1}}
\def\myup#1{{\mytensor}^{\vphantom{\mysups}#1}}
\def\tensor#1#2{
  #1
  \def\mytensor{\vphantom{#1}}
  \def\mysubs{\relax}
  \def\mysups{\relax}
  \let\down=\mysavedown
  \let\up=\mysaveup
  #2
  \let\down=\mydown
  \let\up=\myup
  #2
  }

\newcommand{\Hess}{\operatorname{Hess}}

\newcommand{\Tr}{\operatorname{Tr}}

\newcommand{\R}{\mathbb R}

\newcommand{\mbbS}{\mathbb S}

\renewcommand{\div}{\operatorname{div}}
\renewcommand{\epsilon}{\varepsilon}
\renewcommand{\hat}{\widehat}

\def\crn#1#2{{\vcenter{\vbox{
        \hbox{\kern#2pt \vrule width.#2pt height#1pt
           }
          \hrule height.#2pt}}}}

\newcommand{\U}{\mathbb U}

\renewcommand{\hbar}{{\overline h}}

\newcommand{\pre}[2]{{{\vphantom{#2}}^{#1}}\kern-.2ex{#2}}

\theoremstyle{plain}
\newtheorem{theorem}{\sc Theorem}[section]

\newtheorem{proposition}[theorem]{\sc Proposition}

\theoremstyle{definition}
\newtheorem{remark}[theorem]{\sc  Remark\rm}

\numberwithin{equation}{section}

\date{March 26, 2010}%

\begin{document}

\title[ Gluing  metrics in interpolating their scalar curvature] {Localized gluing of Riemannian metrics in
interpolating their scalar curvature}
\author[E.
Delay]{Erwann Delay} \address{Erwann Delay, Laboratoire d'analyse
non lin\'eaire et g\'eom\'etrie, Facult\'e des Sciences, 33 rue
Louis Pasteur, F84000 Avignon, France}
\email{Erwann.Delay@univ-avignon.fr}
\urladdr{http://www.math.univ-avignon.fr/Delay}
\begin{abstract}
We show that two smooth nearby Riemannian metrics can be glued
interpolating their scalar curvature. The resulting smooth
metric is the same as the starting ones outside the gluing
region and has scalar curvature interpolating between the
original ones. One can then glue metrics while maintaining
inequalities satisfied by the scalar curvature.
We also glue asymptotically Euclidean metrics to Schwarzschild
ones and the same for asymptotically Delaunay metrics, keeping
bounds on the scalar curvature, if any. This extend the Corvino
gluing near infinity to non-constant scalar curvature metrics.
\end{abstract}

\maketitle



\tableofcontents
\section{Introduction}\label{section:intro}
The Corvino-Schoen method enables gluing near infinity constant
scalar curvature metrics (or more generally relativistic initial
data)  to a Schwarzschild (or Kerr) type model. This method was used
in many contexts and has a lot of very nice applications
\cite{Corvino}, \cite{CorvSchoen}, \cite{CD1} \cite{CD2},
\cite{ChrPol}, \cite{ChrPacPol} \cite{CD5}, \cite{ChrCorvIsen},...

It is now natural to see how far this approach can be extended. In
\cite{Delay:ColleP} the author shows that the method works for a
large class of underdetermined elliptic operators of any order. In
the present study, we will see that the gluing can be done if we
substitute the assumption
 that the scalar curvature interpolates between
the starting metric and a model, for the constant scalar curvature
assumption. We then keep the bound on the scalar curvature if there is
one.  We also recover  a constant scalar curvature
metric if we start with such a metric. \\

Let $(M,g)$ be a smooth Riemannian manifold. \red{We do not
assume that $(M,g)$ is connected nor complete nor compact}. Let
$\Omega_i$, $i=1,2,3$ be open subsets of $M$ with smooth
boundary and such that $\overline{\Omega}_1\subset
\Omega_2\subset\overline{\Omega}_2\subset \Omega_3$. We set
$\Omega=\Omega_2\backslash \overline{\Omega}_1$ and we assume
that $\overline\Omega$ is compact. Let $\overline g$ be another
smooth Riemannian metric on $\Omega_3\backslash
\overline{\Omega}_1$.

We will use a perturbation argument to glue $\overline g$ with $g$ on $\Omega$, so we are interested in $P_g$,
the linearized scalar curvature operator:
$$
P_gh:=DR(g)h=\div_g \div_g h+\Delta_g \Tr_g h -\langle \Ric(g),h\rangle_g,
$$
where our Laplacian $\Delta=\nabla^*\nabla$ is positive.
The surjectivity of $P_g$ is, at least formally, related to the injectivity of
 its $L^2$ formal adjoint :
$$
P^*_gu=\Hess_g u+\Delta_g u\;g -u\Ric(g).
$$
We will say that the metric $g$ is {\it non degenerate} on
$\Omega$ if the kernel $\mathcal K$ of $P^*_g$ is trivial on
this set. This condition is generic \cite{BCS}.

We now state the local deformation:
\begin{theorem}\label{maintheorem}
 Let $\chi$ be a smooth cutoff function equal to $1$ near $\overline\Omega_1$ and to $0$ near the complementary of
 $\Omega_2$.
 If $g$ is non degenerate and $\overline g$ is close to $g$ on $\overline \Omega$ then there
 exists
 a symmetric covariant two tensor
  $h\in C^\infty({ \Omega_3})$, supported in $\overline\Omega$ such that the metric
 $$
 \widetilde g:=\chi g+(1-\chi)\overline g+h,
 $$
 solves
 $$R(\widetilde g)=\chi R(g)+(1-\chi)R(\overline g).$$
\end{theorem}

This first result is very closely related to Theorem 1 in
\cite{Corvino}, and the proof is almost the same: it is given in
Section \ref{sec:fixe}.\\

We are now interested in asymptotically Euclidean  (AE for
short) metrics (see Section \ref{sec:AE} for the  precise
definition). We will use the procedure to glue these metrics
with a Schwarzschild model (slice), defined on
 $\R^n\backslash \{c\}$ by:
$$
g_S=g_{m,c}=\left(1+\frac m{(n-1)|x-c|^{n-2}}\right)^{\frac 4{n-2}}\delta,
$$
where $\delta$ is the Euclidean  metric. The gluing will occur
on an annulus
$$
A_{\lambda,4\lambda}=\{x\in\R^n, \lambda<r=|x|<4\lambda\}.
$$
The proximity of the metrics will be guaranteed for $\lambda$ large.

\begin{theorem}\label{ThAE}
Let $g$ be an asymptotically Euclidean   metric of order
$\alpha>n/2-1$.  Assume the mass $m_g$ of $g$ is not zero,
assume also $R(g)=O(r^{-\beta})$ with $\beta>n$, and finally
assume that $g$ satisfies the asymptotic parity condition. Then
there exists $\lambda_0>0$ such that for all
$\lambda>\lambda_0$ the metric $g$ can be glued with a
Schwarzschild metric on the annulus $A_{\lambda,4\lambda}$. The
resulting metric $g_\lambda$ has scalar curvature interpolating
between $R(g)$ and $0$ on $A_{\lambda,4\lambda}$. In
particular, if $g$ has  non negative scalar curvature then so
has $g_\lambda$. %In any case, $g_\lambda$ converges to $g$ as
%$\lambda$ goes to infinity
%  and the same is true for the mass
%and  the center of mass.
\end{theorem}

The asymptotic
parity condition  is defined in Section
\ref{sec:AE}, it  ensures  the center of mass to be well defined on an AE chart.\\

Finally, we show in Section \ref{sec:AD} that the method can be
adapted to asymptotically
Delaunay metrics giving Theorem \ref{thAD}.\\

There also exists such a constant scalar curvature gluing in an
asymptotically hyperbolic setting \cite{CD5}. The AH metrics
are then glued to a Schwarzschild AdS (a particular Kottler)
metric on an annulus near the infinity. This result can
certainly be adapted to non constant scalar curvature metrics
as before. However the result there is probably  not sharp, so
 it seems
sensible  to wait for a sharper (and if possible simpler)
version before extending it to the non constant scalar
curvature case.\\

The approach
developed here can certainly
be adapted to the full constraint map \cite{CorvSchoen}
\cite{CD2} and to other operators \cite{Delay:ColleP}, but the
utility is less clear for the moment.

\section{Gluing on a fixed set}\label{sec:fixe}

In this section we give the proof of Theorem \ref{maintheorem}. This
proof is almost the same as in \cite{Corvino} or \cite{CD2} for
instance. We just recall the different steps for
 completeness.

\subsection{Weighted
spaces}\label{SwSs}

We will use the spaces already introduced in \red{the} appendix
of \cite{CD2} in the special case of \red{a} compact boundary.
We keep the general notation of \cite{CD2} for easy comparison
with that paper.

Let $x\in C^{\infty}(\overline{\Omega})$ be a (non negative)
defining function of the boundary
$\partial\Omega=x^{-1}(\{0\})$.

Let $a\in\N$, $s\in\R$, $s>0$ and let us define
$$\phi=x^2\;, \quad
\psi=x^{2(a-n/2)}e^{-s/x}\mbox{ and }  \varphi=x^{2a}e^{-s/x}.$$

For $k\in \N$ let $H^k_{\phi,\psi}$ be the space of $H^k_{{loc}}$
functions or tensor fields such that the norm
 \be \label{defHn}
 \|u\|_{H^k_{\phi,\psi}}:=
(\int_M(\sum_{i=0}^k \phi^{2i}|\nabla^{(i)}
u|^2_g)\psi^2d\mu_g)^{\frac{1}{2}} \ee is finite, where
$\nabla^{(i)}$ stands for the tensor $\underbrace{\nabla ...\nabla
}_{i \mbox{ \scriptsize times}}u$, with $\nabla$ --- the
Levi-Civita covariant derivative of $g$;  For
$k\in \N$ we denote by $\zH^k_{\phi,\psi}$ the closure in $H^k_{\phi,\psi}$ of the
space of $H^k$ functions or tensors which are compactly (up to a
negligible set) supported in $\Omega$, with the norm induced from
$H^k_{\phi,\psi}$.
The $\zH^k_{\phi,\psi}$'s are Hilbert spaces with the obvious scalar product
associated to the norm \eq{defHn}. We will also use the following
notation
$$
\quad \zH^k  :=\zH^k  _{1,1}\;,\quad
L^2_{\psi}:=\zH^0_{1,\psi}=H^0_{1,\psi}\;,
$$ so that $L^2\equiv \zH^0:=\zH^0_{1,1}$. We  set
$$
W^{k,\infty}_{\phi}:=\{u\in W^{k,\infty}_{{{loc}}} \mbox{ such that }
\phi^i|\nabla^{(i)}u|_g\in L^{\infty}\}\;,
$$
with the obvious norm, and with $\nabla^{(i)}u$ --- the
distributional covariant derivatives of $u$.

For  $k\in\N$ and $\alpha\in [0,1]$, we define
$C^{k,\alpha}_{\phi,\varphi}$ the space of $C^{k,\alpha}$
functions or tensor fields  for which the norm
$$
\begin{array}{l}
\|u\|_{C^{k,\alpha}_{\phi,\varphi}}=\sup_{x\in
M}\sum_{i=0}^k\Big(
\|\varphi \phi^i \nabla^{(i)}u(x)\|_g\\
 \hspace{3cm}+\sup_{0\ne d_g(x,y)\le \phi(x)/2}\varphi(x) \phi^{i+\alpha}(x)\frac{\|
\nabla^{(i)}u(x)-\nabla^{(i)}u(y)\|_g}{d^\alpha_g(x,y)}\Big)
\end{array}$$ is finite.

\begin{remark}
In the context of compact boundary, it is more usual to use $\phi=x$ and for $\psi$ and $\varphi$ a power of $x$
which can be done here also as long as we work with finite differentiability. We choose to take the exponential weight
to treat all the case\red{s} in the same way.
\end{remark}

\subsection{The gluing}
We give the proof of Theorem \ref{maintheorem} by three
propositions, giving the different steps needed. Let $g$ be a
smooth fixed Riemannian metric on $\overline{\Omega}$. We
denote by $\mathcal K$ the kernel of $P^*_g$. For any  metric
$\widetilde g$ on $\overline{\Omega}$, we set
$$ {\mathcal L}_{\phi,\psi}(\widetilde g):=
\psi^{-2} P_{\widetilde g}\psi^2\phi^{4} P_{\widetilde g}^*
\;.$$

 We denote by $\maclKzo $ the $L^2_\psi(\widetilde g)$ orthogonal to
 (the fixed set) $\mathcal K$, and   $\pi_{\maclKzo }$ the $L^2_\psi(\widetilde g)$ projection
onto  $\maclKzo $. We   have (see \cite{CD2} for instance)

\begin{proposition}\label{iso}For $k\ge 0$,
 the map \bel{iso} \pi_{\maclKzo } {\mathcal L}_{\phi,\psi}(\widetilde{g}): {\maclKzo }\cap  \Hpsi{k+4}(\widetilde{g})
\longrightarrow {\maclKzo }\cap \Hpsi{k}(\widetilde{g}) \ee is an
isomorphism such that the norm of its inverse is bounded
independently of $\widetilde{g}$ close to $g$ in
$W^{k+4,\infty}_\phi$.
\end{proposition}

We can now state

\begin{proposition}\label{mainprop} Let $k>n/2$.
 Let $\chi$ be a smooth cutoff function equal to $1$ near $\overline\Omega_1$ and to $0$ near the complementary of
 $\Omega_2$. Let us define
 $$g_\chi:=\chi g+(1-\chi)\overline g\mbox{ and }R_\chi:=\chi R(g)+(1-\chi)R(\overline g).$$
 If $\overline g$ is close to $g$ in $C^{k+4}(\overline \Omega)$ then there exists a unique
 $h=\psi^2\phi^{4} P_{g_\chi}^*u$, with $u\in\Hpsi{k+4}(g_\chi)$ such that
 \bel{solmodker} \pi_{\maclKzo }\psi^{-2}[R(g_\chi+h)-R_\chi]=0.\ee
\end{proposition}
\begin{proof} Apply \cite{CD2} theorem 5.9, with $K=Y=J=0$, $N=u$ and $\delta\rho=R_\chi-R(g_\chi)$. In fact we can solve
$$\pi_{\maclKzo }\psi^{-2}[R(g_\chi+h)-R(g_\chi)]=\pi_{\maclKzo }\psi^{-2}[R_\chi-R(g_\chi)]$$
for $g_\chi$ close to $g$ by a uniform inverse function theorem
using Proposition \ref{iso}. Note  that $R_\chi-R(g_\chi)$ vanishes
near the boundary and tends to zero together with $g_\chi-g$ when
$\overline g$ approaches $g$.
\end{proof}
\begin{remark}
The cutoff function $\chi$ used to interpolate the scalar
curvatures can be chosen to be different from the one
interpolating the metrics.
\end{remark}
From Proposition 5.10 and
Corollary 5.11 of \cite{CD2}
 we have more regularity.

\begin{proposition}
Under the conditions of Proposition \ref{mainprop}, assume
moreover that $k\geq [\frac n 2]+1$, and $\overline g$ is close
to $g$ in $C^{k+4,\alpha}(\overline \Omega)$. Then the solution
$h$ of proposition \ref{mainprop} is in $\phi^2
\psi^2C^{k+2,\alpha}_{\phi,\varphi}(\Omega)$. Moreover, if
$\overline g\in C^{\infty}$, then $h\in
\phi^2\psi^2C^{\infty}_{\phi,\varphi}(\Omega)\subset
C^{\infty}(\overline\Omega)$ and can be smoothly extended by
zero across the boundary.
\end{proposition}

This proposition then concludes the proof of Theorem
\ref{maintheorem} which assumes that the metric
 $g$ is non degenerate. We will see in the next sections that the kernel projection can  be removed
 in some circumstances when a kernel is present.

 \subsection{Remark on a regularity improvement}
In the Corvino-Schoen construction, there is a loss of
regularity in the following way. Assume $g$ has regularity of
order $k+4$ and $\overline g$ has regularity $k+2$, the final
glued metrics can only have a regularity of order $k$ because
of the Ricci term in $P^*_{g_\chi}$. In \cite{CD3} we
regularize $ g_\chi$ in order to recover a better
regularity. There is a simpler way that can be used for many
applications such as the ones that follow. We define
$\hat{\mathcal L}(g_\chi)$ as $\mathcal L(g_\chi)$
but by replacing just the $P^*_{g_\chi}$ term by $P^*_{ g}$ (or by $\Hess_{
g_\chi}+\Delta_{ g_\chi}u\; g_\chi-u\Ric(g)$). Clearly, this operator has the same
isomorphism properties as $\mathcal L(g)$ for $\overline g$
close to $g$ with $k+2$ derivatives, with also a uniform
inverse bound. Using this, we then improve the regularity of
the final glued metric. We do not use this for the
applications, in order to avoid unnecessary complications
of the comprehension.

\section{Exactly Schwarzschild end}\label{sec:AE}
We show that the gluing construction of Corvino \cite{Corvino}
for scalar flat AE metrics, can be done without the scalar flat
assumption, by interpolating the scalar curvature of $g$ and
zero. This just needs  a fast decay of the scalar curvature
(the natural decay for $R(g)\in L^1$).

An asymptotically Euclidean  manifold $(M,g)$ of order
$\alpha>0$ is for us a Riemannian manifold with (at least) an
end $E\approx[A,+\infty[\times \mbbS^{n-1}$, $A>0$ such that
there exists a chart of this end on which the components of $g$
satisfy, for multi-indices $\beta$,
$$
\partial^\beta(g_{ij}-\delta_{ij})=O(r^{-\alpha-|\beta|}),
$$
near infinity, where $r=|x|$.

In order to guarantee  existence of a center of mass, we impose
an  asymptotic parity condition. That is, we assume there
exist an AE chart in which
\bel{RTC} |g^-_{ij}|+r|\partial_k(g^-_{ij})|\leq
C(1+r)^{-\alpha_-},\;\;\; \alpha_->\alpha,
\;\;\;\alpha+\alpha_-> (n-1), \ee and \bel{RTCR} |R(g)^-|\leq
C(1+r)^{-\beta_-},\;\;\; \beta_->\beta, \;\;\;\beta_-> n+1, \ee
where $f^-(x)= \frac12 [f(x)-f(-x)]$.
Note that this condition may probably be relaxed \cite{ChrCorvIsen2}.\\

We can now give the proof of Theorem \ref{ThAE}.

\begin{proof}
Let $\phi_\lambda:A_{1,4}\longmapsto A_{\lambda,4\lambda}$ be
defined by
$$
\phi_\lambda(x)=\lambda x.
$$
For any $\alpha$-asymptotically Euclidean  metric $g$ we define
on $A_{1,4}$
$$
g_\lambda=\lambda^{-2}\phi^*_\lambda g.
$$
In particular one has $g_\lambda=\delta+O(\lambda^{-\alpha})$ on
$A_{1,4}$. Now let $g$ as in Theorem \ref{ThAE}. Let $\chi$ be a
smooth non negative function equal to $1$ on $A_{1,2}$ and $0$ on
$A_{3,4}$ Let
$$
\overline{g}_{\lambda,S}=\chi g_\lambda+(1-\chi)g_{\lambda,S}.
$$
The equation we will solve is
\bel{equascal}
R(\overline{g}_{\lambda,S}+h)=\chi R({g}_{\lambda}),
\ee
where $h$ and its derivatives vanish on $\partial A_{1,4}$.
Let
$$
\mathcal K=\mbox{span}\{1,x^1,...x^n\},
$$
the kernel of $P^*_\delta$.
We can use the preceding procedure to solve
$$
\pi_{\mathcal K^\bot}\psi^{-2}[R(\overline{g}_{\lambda,S}+h)-R(\overline{g}_{\lambda,S})]=\pi_{\mathcal K^\bot}\psi^{-2}[\chi R({g}_{\lambda})-R(\overline{g}_{\lambda,S})],
$$
for any $\lambda$ large enough. Note that the correction is of order the error, that is
$h=O(\lambda^{-\gamma})$, $\gamma=\min(\alpha,n-2,\beta-2)=\min(\alpha,n-2)$.\\

We will now prove that we can choose $S=(m,c)$ to kill the kernel
projection. Let $\widetilde
g_{\lambda,S}=\overline{g}_{\lambda,S}+h=\delta+O(\lambda^{-\gamma})$.
The components of the projections onto the kernel are
\begin{eqnarray*}
q^0_\lambda(S):=<1, \psi^{-2}[R(\widetilde{g}_{\lambda,S})-\chi R({g}_{\lambda})]>_{L^2_\psi}&=&\int_{A_{1,4}}[R(\widetilde g_{\lambda,S})-\chi R({g}_{\lambda})]\\
&=&\int_{A_{1,4}}R(\widetilde g_{\lambda,S})+O(\lambda^{2-\beta}),
\end{eqnarray*}
\begin{eqnarray*}
q^l_\lambda(S):=<x^l, \psi^{-2}[R(\widetilde{g}_{\lambda,S})-\chi R({g}_{\lambda})]>_{L^2_\psi}&=&\int_{A_{1,4}}x^l[R(\widetilde{g}_{\lambda,S})-\chi R({g}_{\lambda})]\\
&=&\int_{A_{1,4}}x^lR(\widetilde{g}_{\lambda,S})+O(\lambda^{2-\beta}).
\end{eqnarray*}
We can deduce (see Appendix \ref{appendice} for details, also recall
that $\beta>n$)
\begin{eqnarray*}
q^0_\lambda(S)&=&4\omega_{n-1}\lambda^{2-n}[(m-m_g)+o(1)]+O(\lambda^{2-\beta})\\
&=&4\omega_{n-1}\lambda^{2-n}[(m-m_g)+o(1)],
\end{eqnarray*}
\begin{eqnarray*}
q^l_\lambda(S)&=&4\omega_{n-1}\lambda^{2-n}[C\;^l_{g_S}-C\;^l_g+o(1)]+O(\lambda^{2-\beta})\\
&=&4\omega_{n-1}\lambda^{2-n}[C\;^l_{g_S}-C\;^l_g+o(1)],
\end{eqnarray*}
where the $o(1)$'s are uniform relatively to $(m,C)$ when $(m,C)$
stay in a compact set where $m\neq 0$, say a small closed ball
around $(m_g,C_g)$. From an application to the Brouwer fixed point
theorem (see eg. Lemma 3.18 in \cite{CD2} for details), for any
$\lambda$ large enough, there exist $(m,C)$ (so an $S=(m,c)$) such
that $q_\lambda(S)=0$
\end{proof}
\section{Exactly Delaunay end}\label{sec:AD}
We show that the gluing of asymptotically Delaunay metrics of \cite{ChrPol} \cite{ChrPacPol}
can be extended to non constant scalar curvature metrics.\\

Let $N$ be a smooth compact $(n-1)$ dimensional manifold without boundary.
A (generalized) Delaunay metrics on $\R\times N$ is of the
form
$$
\mathring g=u^{\frac 4{n-2}}(dy^2+\mathring h),
$$
where $\mathring h$ is the Einstein metric on $N$ with scalar
curvature $(n-1)(n-2)$, and $u=u(y)>0$ is a periodic solution of
$$
u''-\frac{(n-2)^2}4u+\frac{n(n-2)}4u^{\frac{n+2}{n-2}}=0.
$$
The metric depends on two parameters, the period of $u$ and the neck
size $\epsilon$ which is the minimum of $u$ (we do not allow the
critical cases of the sphere  and the cylinder ):
$$
0<\epsilon<\left(\frac{n-2}2\right)^{\frac{n-2}4}.
$$
Let $(M,g)$ be a smooth $n$ dimensional Riemannian manifold. We
will say that $M$ is asymptotically Delaunay, if there exists
an end $E\approx[0,+\infty)\times N$, on which  $g$ is
asymptotic to a Delaunay metric $\mathring g=\mathring
g_\epsilon$ together with derivatives up to order four.

On the end $E$ the Delaunay metric can be written of the form
$$
\mathring g_\epsilon=dx^2+e^{2f(x)}\mathring h,
$$
where $f$ has a period $T=T(\epsilon)$. Let $\chi$ be a smooth cut
off function on $[0,T]$ equal to $1$ near zero and zero near $T$,
and let $\chi_i$ be its translated on $[iT+\sigma,(i+1)T+\sigma]$
for a fixed $\sigma$ to be made precise later. Let
$\Omega_i=[iT+\sigma,(i+1)T+\sigma]\times N$.

We then have (compare \cite{ChrPol}, Theorem 3.1)
\begin{theorem}\label{thAD}
Assume that $(M,g)$ is asymptotic to a Delaunay metric $\mathring
g_\epsilon$ on an end $E$. Then for  $i$ large enough, there exist
$\epsilon'$ and a metric $g_i$ on $E$ which coincide with $g$ before
$\Omega_i$ and with a
 Delaunay $\mathring g_{\epsilon'}$ after $\Omega_i$ and such that $$R(g_i)=\chi_i R(g)+(1-\chi_i)n(n-1).$$
 %Moreover $g_i$ tends to $g$ and $\epsilon'$ tends to $\epsilon$ when $i$ goes to infinity.
\end{theorem}

\begin{proof}
We just mention the changes needed in the proof of Theorem 3.1 in
\cite{ChrPol}. We define $g_{\epsilon'}=\chi_i g+(1-\chi_i)\mathring
g_{\epsilon'}$ and $R_{\epsilon'}=\chi_i R(g)+(1-\chi_i)n(n-1)$. We
can then solve on $\Omega_i$, for $i$ large enough the equation
$$
\pi_{\mathcal K^\bot}\psi_i^{-2}[R(g_{\epsilon'}+h_i)-R(g_{\epsilon'})]=\pi_{\mathcal K^\bot}\psi_i^{-2}[ R_{\epsilon'}-R(g_{\epsilon'})],
$$
where $\psi_i$ is the translation on $\Omega_i$ of $\psi$
defined on $[0,T]$ as before. Here $\mathcal K$ is one
dimensional and spanned by the $\mathring N$ of  \cite{ChrPol}.
We also call the resulting metric $\widetilde
g_{\epsilon'}:=g_{\epsilon'}+h_i$. Here also the correction is
of order the perturbation introduced
$h_i=O(|\epsilon'-\epsilon|)+o_i(1)$, where $o_i(1)$ tends to
zero uniformly relatively to $\epsilon'$ bounded when $i$ tends
to infinity. The projection onto the kernel becomes (see
equation (3.11) in \cite{ChrPol})\footnote{The measure used in
\cite{ChrPol} equation (3.11) appears to be the wrong one when
using the theorem of \cite{CD2} as described there, this does not affect
their proof, just a small change in equation (3.19)
there is needed.}
$$
q_i(\epsilon'):=\int_{\Omega_i}\mathring N[R(\widetilde g_{\epsilon'})-R_{\epsilon'}]d\mu_{g_{\epsilon'}}.
$$
This projection can be rewritten as
\begin{eqnarray*}
q_i(\epsilon')&=&\int_{\Omega_i}\mathring N[R(\widetilde g_{\epsilon'})-n(n-1)]d\mu_{g_{\epsilon'}}
+\int_{\Omega_i}\mathring N[n(n-1)-R_{\epsilon'}]d\mu_{g_{\epsilon'}}\\
&=&\int_{\Omega_i}\mathring N[R(\widetilde g_{\epsilon'})-n(n-1)]d\mu_{ g_{\epsilon'}}+o_i(1)\\
&=&\int_{\Omega_i}\mathring N[R(\widetilde g_{\epsilon'})-n(n-1)]d\mu_{ \widetilde g_{\epsilon'}}+O((|\epsilon'-\epsilon|+o_i(1))^2)+o_i(1).
\end{eqnarray*}
We can then proceed as in Equation (3.19) in \cite{ChrPol} (with the
missprints of the boundary and the measure corrected):
\begin{eqnarray*}
q_i(\epsilon')&=&\int_{\{(i+1)T+\sigma\}\times N}\U^idS_i-\int_{\{iT+\sigma\}\times N}\U^idS_i+O((|\epsilon'-\epsilon|+o_i(1))^2)+o_i(1)\\
&=&\lambda(m'-\mathring m)+O((\epsilon'-\epsilon)^2)+O(|\epsilon'-\epsilon|) o_i(1)+ o_i(1),
\end{eqnarray*}
where we recall that $\sigma$ has been chosen such that
$$
\lambda:=4\omega_{n-1}(n-1)\mathring N_{|_{x=\sigma}}\neq0.
$$
By the intermediate value theorem, for any $i$ large enough, we can choose $m'$ (equivalently $\epsilon'$)
such that $q_i(\epsilon')=0$.

\end{proof}

\section{Appendix: Scalar curvature, mass and center of mass}\label{appendice}
Here we recall the basic definitions and relations between the objects of the title in the AE setting. \\

Let $g$ be an asymptotically euclidian metric of order
$\alpha>n/2-1$. Assume also that the asymptotic parity condition
\eq{RTC} and \eq{RTCR} holds. In this section, we will lower or raise
indices with the Euclidean  metric $\delta$ and its inverse.
 Let us define
$$
V_{ij}:=g_{ij}-(\Tr_\delta g)\;\delta_{ij}.
$$
The mass of $g$ is defined  by
$$
m_g:=\lim_{\lambda\rightarrow\infty}m_g(\lambda)\;,\;\;\;
m_g(\lambda):=\frac1{4\omega_{n-1}}\int_{\mbbS_\lambda}(\partial^iV_{ij})\nu^jds,
$$
where $\nu^j=x^j/|x|$ is the unit normal,   $ds$
is the standard  measure  on the sphere   $\mbbS_\lambda$ of radius $\lambda$ and $\omega_{n-1}$ is the volume
of the unit sphere in $\R^n$. The center of mass is defined
  by
$$
C\;^l_g=(mc)_g^l:=\lim_{\lambda\rightarrow\infty}C\;^l_g(\lambda)\;,\;\;\;
C\;^l_g(\lambda):=\frac1{4\omega_{n-1}}\int_{\mbbS_\lambda}[x^l\partial^iV_{ij}-V_j^l]\nu^jds.
$$
%where $c_n=-\frac{n}{4(n-2)\omega_{n-1}}$.

The scalar curvature satisfies near infinity,
$$
R(g)\sqrt{|g|}=\partial^j\partial^iV_{ij}+Q,
$$
where $|g|$ is the determinant of $g$, $e=g-\delta$, and $Q$ is quadratic in  $\partial e=O(r^{-(\alpha+1)})$.
In particular we see that,
$$
\frac1{4\omega_{n-1}}\int_{A_{\lambda,\nu}}R(g)d\mu_g=m_g(\lambda)-m_g(\nu)+O(\lambda^{-2\alpha+n-2}),
$$
and when $R(g)=O(r^{-\beta})$, $\beta>n$, we have (for $\nu=+\infty$ ):
$$
m_g(\lambda)=m_g+O(\lambda^{-2\alpha+n-2})+O(\lambda^{-\beta+n})=m_g+o(1).
$$
In the same way, and using parity considerations (just write $e=e^++e^-$), we see that,
$$
\frac1{4\omega_{n-1}}\int_{A_{\lambda,\nu}}x^lR(g)d\mu_g=C\;^l_g(\lambda)-C\;^l_g(\nu)+O(\lambda^{n-1-\alpha-\alpha_-}).
$$
When $R(g)^-=O(r^{-\beta_-})$, $\beta_->n+1$,
we have (for $\nu=+\infty$) :
$$
C\;^l_g(\lambda)=C\;^l_g+O(\lambda^{n-1-\alpha-\alpha_-})+O(\lambda^{n+1-\beta_-})=C\;^l_g+o(1).
$$
\\

As the computations needed are made on a rescaled annulus we
give the simple relations between the scalar curvatures and the
measures. Let $g$ be  $\alpha$-AE and $g_\lambda$  as in
Section \ref{sec:AE}, for $x\in A_{1,4}$
  let $y=\phi_\lambda(x)$ so
$$
R(g_\lambda)(x)=\lambda^2R(g)(y),
$$
and
$$
\sqrt{|g_\lambda|}(x)dx=\lambda^{-n}\sqrt{|g|}(y)dy=(1+O(\lambda^{-\alpha}))\lambda^{-n}dy.
$$
Thus one obtains \bel{Rdmu}
R(g_\lambda)(x)\sqrt{|g_\lambda|}(x)dx=\lambda^{2-n}R(g)(y)\sqrt{|g|}(y)dy.
\ee

\bigskip

\noindent{\sc Acknowledgements} :
I am grateful to P. T. Chru\'sciel and F. Gautero for their comments on the original manuscript.

\bibliographystyle{amsplain}

\bibliography{../references/newbiblio,%
../references/reffile,%
../references/bibl,%
../references/hip_bib,%
../references/newbib,%
../references/erwbiblio,%
../references/PDE,%
../references/netbiblio,%
stationary,collescal2}

\end{document}